\documentclass[a4paper]{amsart}
\usepackage[12pt]{extsizes}
\usepackage{amsfonts,amsmath,amsthm,amscd,amssymb,latexsym,mathrsfs,enumitem,tikz}
\usepackage[english]{babel}
\usetikzlibrary{positioning,arrows,shapes,shadows}
\tikzset{
    >=stealth',
    punkt/.style={
           circle,
           draw=black,
           text centered},
    punkt2/.style={
           rectangle,
           rounded corners,
           draw=black,
           text centered},
    punkt3/.style={
           rectangle,
           rounded corners,
           draw=black, thick, dashed,
           text centered},
    pil/.style={
           ->,
           shorten <=2pt,
           shorten >=2pt,}
}

\usepackage{color}
\usepackage{geometry}
\theoremstyle{plain}
\newtheorem{theorem}{Theorem}[section]
\newtheorem{lemma}[theorem]{Lemma}
\newtheorem{proposition}[theorem]{Proposition}
\newtheorem{corollary}[theorem]{Corollary}

\theoremstyle{definition}
\newtheorem{definition}[theorem]{Definition}
\newtheorem{example}[theorem]{Example}
\newtheorem{problem}[theorem]{Problem}

\theoremstyle{remark}
\newtheorem{remark}[theorem]{Remark}

\numberwithin{equation}{section}

\newcommand{\N}{\mathbb N}

\newcommand{\R}{\mathbb R}

\newcommand{\mc }{\mathcal}

\author{Marek Balcerzak}
\address{Institute of Mathematics, Lodz University of Technology, al. Politechniki 8, 93-590
\L\'od\'z, Poland}
\email{marek.balcerzak@p.lodz.pl}

\author{\v{L}ubica Hol\'{a}}
\address{Academiy of Sciences, Institute of Mathematics, \v{S}tef\'anikova 49, 814 73 Bratislava, Slovakia}
\email{hola@mat.savba.sk}

\author{Olena Karlova}
\address{Department of Exact and Natural Sciences
		Jan Kochanowski University in Kielce, ul. Uniwersytecka 7, 
    	25-406 Kielce,
        Poland\newline     
    Yurii Fedkovych Chernivtsi National University, 
		 str. Kotsyubynskoho 2,
	     58000 Chernivtsi
         Ukraine}         
\email{okarlova@ujk.edu.pl}

\author{Piotr Szuca}
\address{Institute of Mathematics, Faculty of Mathematics, Physics and Informatics,
University of Gda\'{n}sk, 80-308 Gda\'{n}sk, Poland}
\email{piotr.szuca@ug.edu.pl}

\title{Borel 1 type mappings and the respective equi-families}
\subjclass[2020]{54H05, 26A21}
\keywords{Borel~1 functions, Lebesgue property, weakly separated functions, equi-Baire~1 families, topology of pointwise convergence, functions in two variables of Borel class $\alpha$}
\date{}

\begin{document}

\begin{abstract}
We investigate classes of functions from a topological space to a metric space
that are related to those of Borel class~1.
Following the idea defining an equi-Baire~1 family (due to Lecomte) we define the respective equi-families of functions from the considered classes. We observe that studying of equi-families can be reduced to the exploration of a single orbit map with values in a product space.
We consider the closure of equi-families  with respect to the topology of pointwise convergence. 
Finally, we investigate functions $f\colon X\times Y\to Z$, for metric spaces $X,Y,Z$, with sections that are equi-continuous, equi-Baire~1 or have equi-generalized Lebesgue property with respect to measurable sets of class $\alpha$. In particular, we generalize a result of Grande.
\end{abstract}
\maketitle


\section{Introduction}

Let $f$ be a function between two metric spaces $X$ and $Y$. We say that $f$ is \emph{Baire~1} if $f$ is the limit of a pointwise convergent sequence of continuous functions $f_n\colon X\to Y$, $n\in\N$. We say that $f$ is \emph{Borel~1} if it is $F_\sigma$-measurable, that is, its preimage of every open set is $F_\sigma$. Note that every Baire~1 function is Borel~1, but the converse does not hold in general; see \cite[24B]{Kech}.
However, if $Y=\R$ or $Y$ is a Banach space, then the Baire~1 functions and the Borel~1 functions coincide; see \cite{Kech} and also \cite{Fo} and \cite[Remark 5.6]{Sp} for 
more general results. Note that several authors (see, e.g. \cite{Kech}, \cite{Le}, \cite{LTZ}) use the name Baire~1 for $F_\sigma$-measurable functions which can sometimes lead to confusion. 

Many nice characterizations of Baire~1 (or Borel~1) functions are known. Let us mention some of them. The classical theorem of Baire states that, if $X$ is complete and $Y$ is separable, then $f$ is Borel~1 if and only if
$f\upharpoonright P$ has a point of continuity for every non-empty closed set $P\subseteq X$ (see \cite{Ku}); for some generalizations, see \cite{Kou}). This last condition is called the 
\emph{Point of Continuity Property} (PCP). 

Lee, Tang, and Zhao \cite{LTZ} obtained an interesting 
$\varepsilon$-$\delta$ characterization of Borel~1 functions between two Polish spaces.
The respective condition characterizing Borel~1 functions was called \emph{LTZ-property} in \cite{BHH}.
The result of \cite{LTZ} was later generalized in \cite{FC} to the case when $X$ and $Y$ are separable spaces. The authors of \cite {FC} used an old Lebesgue theorem \cite{Leb} in their proof. Although the Lebesgue result was concerned with the case $X=Y=\R$, it remains true if $X$ is arbitrary and $Y$ is separable \cite[Thm 2.1]{BHH}. The respective condition stated in Lebesgue's
result was referred to in \cite{BHH} as the \emph{Lebesgue property}.

Note that Baire 1 real-valued functions and their regular subclasses have important applications in classical theory of real functions (see \cite{Br}) and functional analysis \cite {HOR}. Independently, some other interesting classes of maps related to Baire 1 functions have been studied, also in a general setting where $X$ is a topological space. Namely, Koumoullis \cite{Kou} considered fragmentable maps and functions with PCP. Bouziad \cite{Bou} investigated weakly separated maps, which coincide with functions with LTZ-property if $X$ is a metric space. He also used cliquish and fragmentable maps, as well as those having PCP. 

In Section~\ref{sec:properties}
we collect the above-mentioned classes of maps. We introduce a new notion --- the \emph{generalized
Lebesgue property}. We establish inclusions between these families. First, this is done when $X$ is a topological space. Some of the inclusions can be reversed under additional assumptions in $X$, while others  are appropriate, as shown by examples. Second, we consider the case where both spaces $X$ and $Y$ are metric spaces.

Our work is, in some sense, a continuation of \cite{BKS} and \cite{BHH}, which were inspired by the intriguing notion of \emph{equi-Baire~1 families} of functions, proposed independently by Lecomte \cite{Le} and Alikhani-Koopeai \cite{A1}. This notion plays a role in dynamical systems similar to that of equi-continuous families (see \cite{A2}). We observed that other types of equi-families of functions can be of interest
(see \cite{GM}, \cite{Gr}  and \cite{BHH}, where equi-PCP property and equi-Lebesgue property were considered).
We collect such families in Section~\ref{sec:families} and present a simple tool which shows that the study of equi-families can be reduced to the analysis of a single orbit map with values in the product space.
Finally, we solve two problems from \cite{BKS} concerning functions of two variables with sections being equi-continuous or equi-Borel measurable of class $\alpha$, and we obtain a generalization of the result of Grande \cite{Gr}.


\section{Properties related to being a Borel 1 map}	
\label{sec:properties}

We consider several collections of functions from a topological space $X$ into a metric space $Y$ that are related to Borel~1 functions and which turned out to be important when studying families of equi-Baire~1 functions.
We establish inclusions between them and show examples when inclusions are proper. If $X$ is a metric space, some of these families are equal under or without additional assumptions on $X$.
	\subsection{Relations between various families of functions}
	
	\begin{definition}\label{OneF}
		Let $X$ be a topological space, $(Y,d)$ be a metric space and $f\colon X\to Y$ be a function. 
		We say that $f$
		\begin{itemize}
			\item is {\it weakly separated}, if 
			\begin{itemize}
				\item[{\rm (WS)}] for every $\varepsilon>0$ there exists a neighborhood assignment $\{V_x\}_{x\in X}$ such that for all $x,y\in X$
				\begin{equation*} 
					(x,y)\in V_{y}\times V_{x} \,\,\,\Longrightarrow\,\,\, d(f(x),f(y))<\varepsilon;
				\end{equation*} 
				(A family $\{V_x\}_{x\in X}$ of open sets in $X$  is called a \emph{neighborhood assignment} whenever $x\in V_x$ for each $x\in X$.) 
			\end{itemize} 
			
			\item is {\it fragmentable}, if 
			\begin{itemize}
				\item[{\rm (Frag)}] for every  $\varepsilon>0$ and every closed set $\emptyset\ne F\subseteq X$ there exists an open set $U$ such that $U\cap F\neq\emptyset$ and ${\rm diam} f(U\cap F)<\varepsilon$;
			\end{itemize}
			
		    \item is {\it cliquish}, if  
		    \begin{itemize}
		    	\item[{\rm (Cliq)}] for every  $\varepsilon>0$ and every  open set $\emptyset\ne U\subseteq X$ there exists an open set $\emptyset\ne O\subseteq U$ such that  ${\rm diam} f(O)<\varepsilon$;
		    \end{itemize}
            (Clearly, $f$ is fragmentable if and only if $f\restriction_F$ is cliquish for each nonempty closed set $F$.)
			
			\item has the \emph{Lebesgue property}, if 
			\begin{itemize}
				\item[{\rm (LP)}] for every $\varepsilon >0$ there is a cover 
				$(X_n)_{n\in\mathbb N}$ of $X$ consisting of closed sets such that ${\rm diam}f(X_n)\le\varepsilon$ for all $n\in\mathbb N$;
			\end{itemize}
						
			\item has the \emph{generalized Lebesgue property}, if
			\begin{itemize}
				\item[{\rm (GLP)}] for every $\varepsilon >0$ there is a $\sigma$-discrete family $\mathcal A_\varepsilon$ consisting of closed subsets of $X$ such that $X=\bigcup \mathcal A_\varepsilon$ and ${\rm diam}f(A)\le\varepsilon$ for all $A\in\mathcal A_\varepsilon$.
			\end{itemize}
			(Recall that a family $\mathcal A$ of subsets of topological space $X$ is {\it discrete}, if every point $x\in X$ has a neighborhood $U$ which intersects at most one set from the family $\mathcal A$. It is worth noting that the union of an arbitrary discrete family of closed sets is also a closed set.
			A family $\mathcal A$ is called {\it $\sigma$-discrete}, if it is a countable union of discrete families.)

			\item has the {\it point of continuity property (PCP)}, if 
			\begin{itemize}
				\item [{\rm (PCP)}]	for every  closed set $\emptyset\ne F\subseteq X$ the restriction $f\restriction_F$ has a point of continuity;
			\end{itemize}

            \item is {\it pointwise discontinuous (PWD)}, if 
			\begin{itemize}
				\item [{\rm (PWD)}]	for every open set $\emptyset\ne U\subseteq X$ the restriction $f\restriction_U$ has a point of continuity;
			\end{itemize}
            (In other words, $f$ is pointwise discontinuous whenever
            the set of continuity points of $f$ is dense in its domain.)
			
			\item is \emph{Borel 1}, if 
			\begin{itemize}
				\item[{\rm (Borel 1)}] the preimage $f^{-1}(V)$ of any open set $V\subseteq Y$ is an $F_\sigma$-set in $X$.
			\end{itemize}							
		\end{itemize}
	\end{definition}

	 \begin{theorem}\label{thm:topX}
	 	  For any topological space $X$ and a metric space $Y$ the following relations hold:
	 		\tikzset{every picture/.style={line width=0.75pt}} 
		
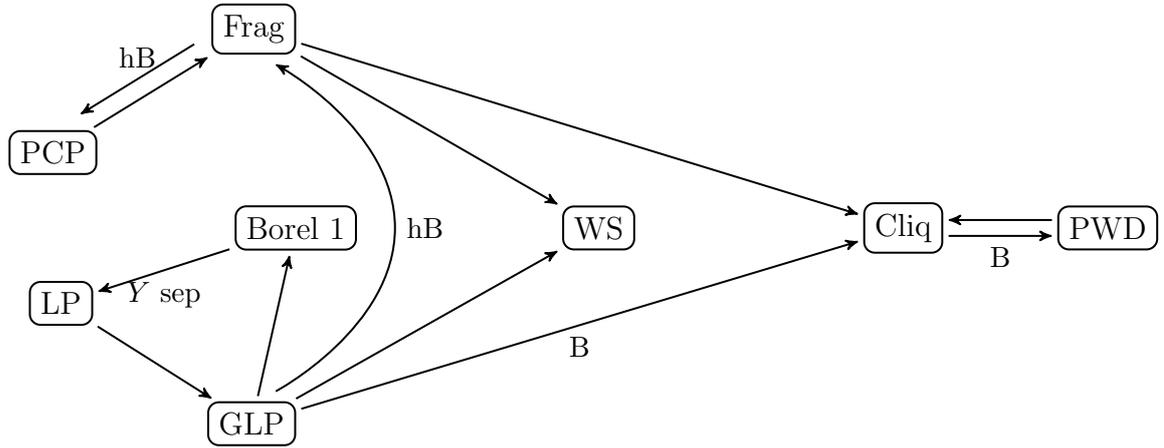
		\begin{center}	

		\begin{tikzpicture}[x=0.75pt,y=0.75pt,yscale=-1,xscale=1,
	node distance=1cm, 
	auto,
	shorter/.style={shorten <=1mm,shorten >=0.5mm},
	|*/.style={to path=(\tikztostart.south) -- (\tikztostart.south|-\tikztotarget.north)},
	*|/.style={to path=(\tikztostart.south-|\tikztotarget.north) -- (\tikztotarget.north)}
 ]
 \node[punkt2] (ws) {\rm WS};
 \node[punkt2,right=3cm of ws] (cliq) {\rm Cliq};
 \node[punkt2,right=1.5cm of cliq] (pwd) {\rm PWD};
 \node[punkt2,above left=2cm and 3.5cm of ws] (frag) {\rm Frag};
 \node[punkt2,below left=2cm and 3.5cm of ws] (glp) {\rm GLP};
 \node[punkt2,left=2.7cm of ws] (b1) {\rm Borel 1};
 \node[punkt2,above left=1cm and 1.5cm of glp] (lp) {\rm LP};
 \node[punkt2,below left=1cm and 1.5cm of frag] (pcp) {\rm PCP};
 \path (frag) edge[pil,->] (cliq);
 \path (frag) edge[pil,->] (ws);
 \path[transform canvas={shift={(-5pt,5pt)}}] (frag) edge[pil,->] node[anchor=south]{{\rm\small hB}} (pcp);
 \path (glp) edge[pil,->] node[anchor=north]{{\rm\small B}} (cliq);
 \path (glp) edge[pil,->] (ws);
 \path (lp) edge[pil,->] (glp);
 \path (pcp) edge[pil,->] (frag);
 \path (glp) edge[pil,->] (b1);
 \path (b1) edge[pil,->] node[anchor=north]{{\rm\small $Y$ sep}} (lp);  
 \path[transform canvas={shift={(0pt,-3pt)}}] (cliq) edge[pil,->] node[anchor=north]{{\rm\small B}} (pwd);
  \path[transform canvas={shift={(0pt,3pt)}}] (pwd) edge[pil,->] (cliq);
 \node[right=0.5cm of b1] (x) {\rm\small hB};
 \node[above right=-5pt and -12pt of glp] (glp2) {};
 \node[below right=-5pt and -12pt of frag] (frag2) {};
 \draw[->] plot [smooth, tension=1.1] coordinates {
    (glp2.north) (x.west) (frag2.south) };

\end{tikzpicture}

	\end{center}
Here (``hB'') ``B'' near an arrow means that $X$ is (hereditarily) Baire, ``$Y$ sep'' means that $Y$ is separable.
\end{theorem}

 \begin{proof}	
    Implications {\bf (PCP) $\Rightarrow$ (Frag)} and {\bf (Frag) $\Rightarrow$ (PCP)} for hereditarily Baire spaces
    were proved by Koumoullis in~\cite[Theorem 2.3]{Kou}. 
	Further, Bouziad \cite{Bou} showed that {\bf (Frag) $\Rightarrow$ (WS)}. 
	Implication   {\bf (Frag) $\Rightarrow$ (Cliq)} follows easily from the definitions.
	Actually, Lebesgue proved that  {\bf (LP) $\Rightarrow$ (Borel 1)} (see \cite[p. 375]{Ku} and \cite{BHH}). 
  	Lebesgue proved also that {\bf (Borel 1) $\Rightarrow$ (LP)} (see also \cite[p. 375]{Ku} and \cite[Theorem 2.1]{BHH}) for a separable space $Y$.    
 	Implication {\bf (Cliq) $\Rightarrow$ (PWD)} was proved in~\cite{MR686978}.
    Since every countable family $\mathcal A$ is $\sigma$-discrete, we have {\bf (LP) $\Rightarrow$ (GLP)}. 

 	Let us show that {\bf (GLP) $\Rightarrow$ (WS)}. Assume $f\colon X\to Y$ has the generalized Lebesgue property and fix $\varepsilon>0$. 
 	There exists a closed cover $\mathcal{A}$ of $X$ 
 	such that ${\rm diam}f(A)\le\varepsilon$ for all $A\in\mathcal{A}$, and $\mathcal{A}=\bigcup_{n\in\mathbb{N}}\mathcal{A}_n$ where $\mathcal{A}_n$ is discrete for each $n\in\mathbb{N}$. 	
 	Note that for all $n\in\mathbb{N}$ and  
    $x\in\bigcup\mathcal{A}_n$ there exists a unique $F^n_x\in\mathcal{A}_n$ such that $x\in F^n_x$.
 	Let $X_n:=\bigcup\mathcal{A}_n$ for each $n\in\mathbb{N}$. For every $x\in X$ we put
 	$$n(x):=\min\{n:x\in X_n\} \textup{\ and\ } V_x:=X\setminus \bigcup_{n<n(x)}\left(X_n\setminus\bigcup \{F\in \mathcal{A}_{n(x)}\colon x\notin F\}\right) .$$
 	Then $\{V_x\}_{x\in X}$ is a neighborhood assignment such that $V_x\cap X_{n(x)}\subseteq F^{n(x)}_x$ for each $x\in X$.
 	Assume $(x,y)\in V_y\times V_x$. Then $x\in X_{n(x)}$ and $x\not\in X_n$ for all $n<n(y)$, consequently $n(x)\ge n(y)$. 
 	Similarly, $n(y)\ge n(x)$. Hence, $n(x)=n(y)$.
 	Since $x\in V_y\subseteq F^{n(y)}_y=F^{n(x)}_y$, we have
 	$F^{n(x)}_y=F^{n(x)}_x$.
 	Thus $d(f(x),f(y))\le {\rm diam}f( F^{n(y)}_y )\le\varepsilon$. 
 	
 	We show that {\bf (GLP) $\Rightarrow$ (Borel 1)}. Let $f\colon X\to Y$ have the generalized Lebesgue property. 
 	For every $n\in\mathbb N$ we choose a  family $\mathcal A_n$ of subsets of $X$ such that 
 	$\mathcal A_n=\bigcup_{k\in\mathbb N}\mathcal A_{n,k}$, 
 	where each $\mathcal A_{n,k}$ is a discrete family consisting of closed sets, 
 	$\mathcal A_n$ is a covering of $X$
 	and ${\rm diam}f(A)<1/n$ for each $A\in \mathcal A_n$. 
 	Denote $\mathcal B:=\bigcup_{n,k\in\mathbb N}\mathcal A_{n,k}$. 
 	Notice that for every subfamily $\mathcal B'\subseteq \mathcal B$ the union $\bigcup\mathcal B'$ of all members in $\mathcal B'$ is an $F_\sigma$-set, 
 	since every $\mathcal A_{k,n}$ is a discrete family of closed sets. 
 	It is easy to see that 
 	$$f^{-1}(V)=\bigcup\{B\in\mathcal B: B\subseteq f^{-1}(V)\} \textup{\ for any open set\ }V\subseteq Y.$$
 	Hence, $f^{-1}(V)$ is an $F_\sigma$-set in $X$. Therefore, $f$ is Borel 1. 
    
  	Finally, we show that {\bf (GLP) $\Rightarrow$ (Cliq)} if $X$ is a Baire space (see also \cite[proof of Theorem 3.8]{H}).  
  	Let $f\colon X\to Y$ have the generalized Lebesgue property, $\varepsilon>0$ and $\emptyset\ne U\subseteq X$ be an open set. 
  	Similarly as in the proof of ``{\bf (GLP) $\Rightarrow$ (WS)}'' there exists a closed cover $\mathcal{A}$ of $X$ 
 	such that ${\rm diam}f(A)\le\varepsilon$ for all $A\in\mathcal{A}$, and $\mathcal{A}=\bigcup_{n\in\mathbb{N}}\mathcal{A}_n$, 
 	where $\mathcal{A}_n$ is discrete for each $n\in\mathbb{N}$. 	
 	Note that for all $n\in\mathbb{N}$ and 
    $x\in\bigcup\mathcal{A}_n$ there exists a unique $F^n_x\in\mathcal{A}_n$ such that $x\in F^n_x$.
 	Let $X_n:=\bigcup\mathcal{A}_n$ for each $n\in\mathbb{N}$.
  	Note that $(X_n)_{n\in\mathbb N}$ is a cover of $X$ consisting of closed sets. 
  	Since $U=\bigcup_{n\in\mathbb{N}} (X_n\cap U)$ and $U$ is a Baire space, there is a nonempty open set $V\subseteq U$ and $n_0\in\mathbb N$ 
  	such that $V\subseteq X_{n_0}\cap U$. 
  	Fix $x\in V$. Since $\mathcal{A}_{n_0}$ is discrete, there exists a nonempty open set $V'$ such that $V'\cap X_{n_0}=V'\cap F^x_{n_0}$.
  	Then ${\rm diam} f(V\cap V')\le\varepsilon$.

    Since the restriction of a $\sigma$-discrete family to a closed set is also $\sigma$-discrete, we obtain, as a corollary of the implication {\bf (GLP)$\Rightarrow$(Cliq)} for Baire spaces, the implication {\bf (GLP)$\Rightarrow$(Frag)} for hereditarily Baire spaces.
 \end{proof}

 \begin{definition}
 	We say that a function $f\colon X\to Y$ between a metric space $(X,\varrho)$ and a metric space $(Y,d)$ has the \emph{Lee-Tang-Zhao (LTZ) property}, if 
	\begin{itemize}
		\item[{\rm (LTZ)}] for every $\varepsilon>0$ there is a positive function $\delta_\varepsilon\colon X\to\mathbb R_+$ such that for all $x,y\in X$
		\begin{equation*} 
			\varrho(x,y)<\min\left\{\delta_\varepsilon(x),\delta_\varepsilon(y)\right\} \,\,\,\Longrightarrow\,\,\,  d(f(x),f(y))<\varepsilon.
		\end{equation*}
	\end{itemize}
 \end{definition}
 
 \begin{theorem}\label{thm:metricX}
 	For any metric spaces $X$ and $Y$ the following relations hold 
 	
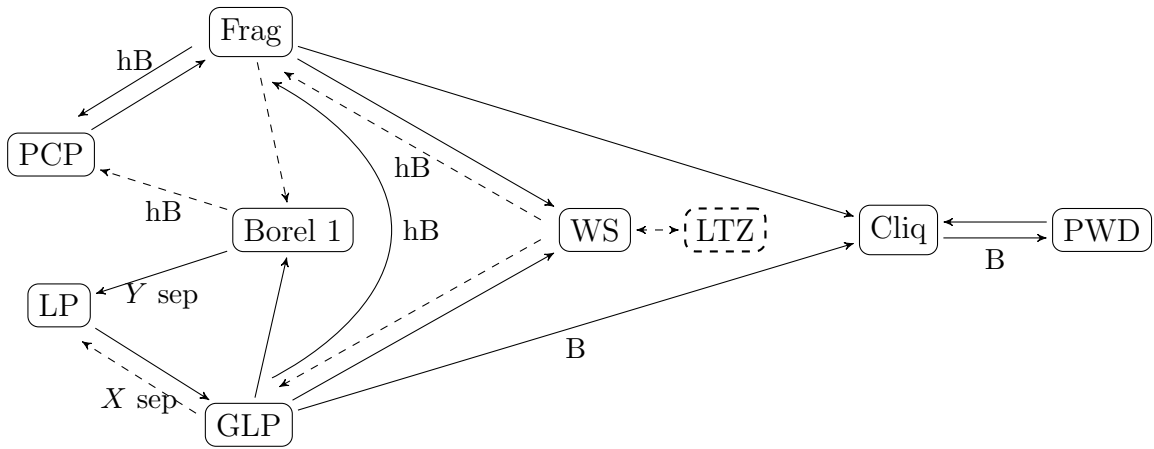
\begin{center}
 	
		\begin{tikzpicture}[x=0.75pt,y=0.75pt,yscale=-1,xscale=1,
	node distance=1cm, 
	auto,
	shorter/.style={shorten <=1mm,shorten >=0.5mm},
	|*/.style={to path=(\tikztostart.south) -- (\tikztostart.south|-\tikztotarget.north)},
	*|/.style={to path=(\tikztostart.south-|\tikztotarget.north) -- (\tikztotarget.north)}
 ]
 \node[punkt2] (ws) {\rm WS};
 \node[punkt2,right=3cm of ws] (cliq) {\rm Cliq};
 \node[punkt2,right=1.5cm of cliq] (pwd) {\rm PWD};
 \node[punkt2,above left=2cm and 3.5cm of ws] (frag) {\rm Frag};
 \node[punkt2,below left=2cm and 3.5cm of ws] (glp) {\rm GLP};
 \node[punkt2,left=2.7cm of ws] (b1) {\rm Borel 1};
 \node[punkt2,above left=1cm and 1.5cm of glp] (lp) {\rm LP};
 \node[punkt2,below left=1cm and 1.5cm of frag] (pcp) {\rm PCP};
 \path (frag) edge[pil,->] (cliq);
 \path (frag) edge[pil,->] (ws);
 \path[transform canvas={shift={(-5pt,5pt)}}] (frag) edge[pil,->] node[anchor=south]{{\rm\small hB}} (pcp);
 \path (glp) edge[pil,->] node[anchor=north]{{\rm\small B}} (cliq);
 \path (glp) edge[pil,->] (ws);
 \path (lp) edge[pil,->] (glp);
 \path (pcp) edge[pil,->] (frag);
 \path (glp) edge[pil,->] (b1);
 \path (b1) edge[pil,->] node[anchor=north]{{\rm\small $Y$ sep}} (lp);  
 \path[transform canvas={shift={(0pt,-3pt)}}] (cliq) edge[pil,->] node[anchor=north]{{\rm\small B}} (pwd);
  \path[transform canvas={shift={(0pt,3pt)}}] (pwd) edge[pil,->] (cliq);
 \node[right=0.5cm of b1] (x) {\rm\small hB};
 \node[above right=1pt and -12pt of glp] (glp2) {};
 \node[below right=1pt and -12pt of frag] (frag2) {};
 \draw[->] plot [smooth, tension=1.1] coordinates {
    (glp2.north) (x.west) (frag2.south) };
  
  \path[transform canvas={shift={(-5pt,-5pt)}}] (ws) edge[pil,->,dashed] node[anchor=north]{{\rm\small hB}} (frag);
  \path[transform canvas={shift={(-5pt,-5pt)}}] (glp) edge[pil,->,dashed] node[anchor=north]{{\rm\small $X$ sep}} (lp);
  \path[transform canvas={shift={(-5pt,5pt)}}] (ws) edge[pil,->,dashed] (glp);
  \path (b1) edge[pil,->,dashed] node[anchor=north]{{\rm\small hB}} (pcp);
  \path (frag) edge[pil,->,dashed] (b1);

  \node[punkt3,right=0.7cm of ws] (ltz) {\rm LTZ};
  \path (ws) edge[pil,<->,dashed] (ltz);  
\end{tikzpicture}
    
 \end{center}
Here (``hB'') ``B'' near an arrow means that $X$ is (hereditarily) Baire, ``$X$ sep'' (``$Y$ sep'') means that $X$ (respectively, $Y$) is separable, dashed line means that $X$ is metric and solid line means that $X$ is arbitrary topological
(solid part of the diagram is covered by Theorem~\ref{thm:topX}). 
 \end{theorem}
 
 \begin{proof} 
 Assume (where it is indicated in the diagram) that $X$ is hereditarily Baire. 
 Condition {\bf (WS) $\Leftrightarrow$ (LTZ)} is an obvious observation;
 it was explicitly stated in introduction to~\cite{Bou}.

 	Implications {\bf (Frag) $\Rightarrow$ (Borel 1)} and {\bf (Borel 1) $\Rightarrow$ (PCP)} were proved by Koumoullis \cite[Theorem 4.12]{Kou}. 
 	 	
 	From the proof of Lee, Tang and Zhao \cite{LTZ} it follows that 
    {\bf (LTZ) $\Rightarrow$ (PCP)} for a hereditarily Baire $X$. 
 	
 	Implication {\bf (Borel 1) $\Rightarrow$ (WS)} for a hereditarily Baire space $X$ was proved in \cite[Theorem 12]{KM}.

 	Implications {\bf (GLP) $\Rightarrow$ (LTZ)} and {\bf (LTZ) $\Rightarrow$ (GLP)} were proved in \cite{KM} (see Lemma 4 and Lemma 5, respectively). 

 	To see {\bf (GLP) $\Rightarrow$ (LP)} for separable metric spaces $X$ observe that, if $X$ is metric and separable, then a $\sigma$-discrete family must be countable.
    (Recall that a straightforward proof of {\bf (LTZ) $\Rightarrow$ (LP)} for separable $X$ was given in \cite[Theorem 5]{FC}).
 \end{proof}
 
 \begin{remark} 	
 	Implication  {\bf (WS) $\Rightarrow$ (Borel 1)}  for a hereditarily Baire semistratifiable space $X$ was also proved by Bouziad \cite[Theorem 4.5]{Bou} (in fact, he proved that {\bf (WS) $\Rightarrow$ (GLP)}).
 	
 	The equivalence  {\bf (Frag) $\Leftrightarrow$ (GLP)} was proved in \cite[Theorem B]{K2019} for a hereditarily Baire perfect paracompact space $X$. 

    Implication {\bf (GLP) $\Rightarrow$ (LP)} works also for a collectionwise normal separable space. Every paracompact space is collectionwise normal. (For the definition of a collectionwise normal space, see \cite[p. 305]{Eng}.)
 \end{remark}

  \begin{corollary}
 	Let $X$ be a hereditarily Baire metric space and $Y$ is a metric space. Then 
 	\begin{center}{\rm 
 			{\bf 	(GLP) $\Leftrightarrow$  (WS) $\Leftrightarrow$ (LTZ) $\Leftrightarrow$ (PCP) $\Leftrightarrow$ (Borel 1) $\Leftrightarrow$ (Frag).}
 		}
 	\end{center}
 	Moreover, if either $X$ or $Y$ is separable, then
 	\begin{center}{\rm 
 			{\bf 	 (LP) $\Leftrightarrow$ (GLP) $\Leftrightarrow$ (WS) $\Leftrightarrow$ (LTZ) $\Leftrightarrow$ (PCP) $\Leftrightarrow$ (Borel 1) $\Leftrightarrow$ (Frag).}
 		}
 	\end{center}
 \end{corollary}

 \subsection{Examples}
 \label{sec:examples}

 In this section we give counterexamples for implications not included in Theorem~\ref{thm:topX} and~\ref{thm:metricX}.
 Case {\bf (Cliq) $\not\Rightarrow$ (PWD)} and {\bf (Frag) $\not\Rightarrow$ (PCP)}
 are covered by Proposition~\ref{prop:baire-frag} and Proposition~\ref{prop:hbaire-frag}, respectively.

	\begin{example}
		{\bf (WS) $\not\Rightarrow$ (Frag)} even if $X$ is metric.
	\end{example}
	
	\begin{proof}
		Let $X=\mathbb Q$ and  $A,B\subseteq X$ be disjoint sets such that $\overline{A}=\overline{B}=X$. Then the function $f\colon X\to\mathbb R$, $f=\chi_A$, is Borel~1 and weakly separated, but ${\rm diam} f(U)=1$ for every open nonempty set $U\subseteq X$. Hence, $f$ is not fragmentable. 		
		
			Since every function on $\mathbb Q$ is Borel 1, we get	{\bf (Borel 1) $\not\Rightarrow$ (Frag)}.
	\end{proof}

		\begin{example}
		{\bf (Cliq) $\not\Rightarrow$ (Borel 1)} even if $X$ is metric and hereditarily Baire. 
	\end{example}
	
	\begin{proof}
		Let $X=Y=\mathbb R$  and $C\subseteq X$ be the Cantor set. Denote by $A$ the set of all endpoints of $C$ and put $f=\chi_A$. It is easy to check that $f\colon X\to Y$ is cliquish, but $f\restriction_C$ is not Borel 1 because $A$ is not a $G_\delta$-set in $C$. Hence, $f\colon X\to Y$ is not Borel 1. 
		
		As a corollary we obtain {\bf (Cliq) $\not\Rightarrow$ (Frag)}.
	 	\end{proof}
	 	
	 	\begin{example}
	 		{\bf (WS) $\not\Rightarrow$ (Borel 1)} even if $X$ is hereditarily Baire and $Y$ is separable.
	 	\end{example}
	 	
	 	\begin{proof}
	 		Let $X=\mathbb S$ be the Sorgenfrey line and $Y=\mathbb R$. Consider $f\colon X\to Y$, $f=\chi_{\mathbb Q}$. Then $f$ is weakly separated, since it is easy to see that every function defined on $\mathbb S$ is weakly separated. But $f$ is not Borel~1 because $\mathbb Q$ is not a $G_\delta$-set in $X$. 
	 		
	 		Consequently,	{\bf (WS) $\not\Rightarrow$ (GLP)}. Moreover, {\bf (WS) $\not\Rightarrow$ (PCP)} if $X$ is not metrizable. 
	 	\end{proof}
	
			\begin{example}[Martin's Axiom and the negation of the Continuum Hypothesis]
			{\bf (Borel 1) $\not\Rightarrow$ (LP)} even if $X$ is metric and separable.
		\end{example}
		
		\begin{proof}
			Under Martin's Axiom and the negation of the Continuum Hypothesis, there exists an uncountable  set $X\subseteq\mathbb R$ such that each subset of $X$ is $F_\sigma$ in $X$; see \cite[Theorem 9.6]{Buk}. Assume that $Y=X$ is endowed with the discrete metric, i.e. $d(x,y)=1$ if $x\ne y$ and $d(x,y)=0$ if $x=y$.  Then the function $f\colon X\to Y$, $f(x)=x$, is Borel 1. Assume that $f$ has the Lebesgue property. For $\varepsilon=\tfrac{1}{2}$ we choose a closed covering $\{F_n:n\in\mathbb N\}$ of $X$ such that ${\rm diam}f(F_n)<\tfrac{1}{2}$. Then ${\rm diam} f(F_n)=0$ which implies that $f(F_n)=\{y_n\}$ for every $n\in\mathbb N$. Since the sequence $(F_n)$ covers $X$, it follows that $Y$ is countable, a contradiction.
		\end{proof}

 Note that, if $X$ is a nonseparable metric space, then the identity function $\operatorname{Id}\colon X\to X$ is Borel~1 and it does not have (LP); see \cite[Proposition 2.1]{BHH}.
 
 \subsection{More examples: Baire spaces and cliquish/fragmentable  functions} 
 \label{sec:examples2}

In this section, we show in the form of a characterization that some implications
in Theorem~\ref{thm:topX} are not reversible. 

 \begin{lemma} \label{le:lubica} Let $X$ be a topological space. If $X$ is not a Baire space, then there are a nonempty open set $U \subseteq X$ and a fragmentable function $f\colon X \to \Bbb R$ such that $f$ has no point of continuity in $U$ and $f(X) \subseteq \{1/n: n \in \Bbb N\}$. Moreover, if $X$ is perfectly normal, then $f$ has the Lebesgue property.
 \end{lemma}
 
 \begin{proof} If $X$ is not a Baire space, then there is a nonempty open set $U$ in $X$ such that $U$ is of the first Baire category in $X$. Let $\{F_n: n \in \Bbb N\}$ be a sequence of relatively closed sets in $U$, which are nowhere dense in $X$ and such that $F_n \subseteq F_{n+1}$ for every $n \in \Bbb N$, and  $U = \bigcup_{n \in \Bbb N} F_n$.
 	
    Define a function $f\colon X\to\mathbb{R}$ by the formula:
      $$f(x)=\left\{\begin{array}{ll}
        1 & \textup{if\ }x \in F_1  \cup (X \setminus U),\\
        \frac{1}{n} & \textup{if\ }x\in F_n  \setminus F_{n-1}, n\in\mathbb{N}, n>1.
        \end{array}\right.$$
    It is easy to verify that there is  no point of continuity of the function $f$ in the open set $U$. (If $n > 1$, then $f^{-1}(\{1/n\}) = f^{-1}((1/(n+1),1/(n-1)) = F_n  \setminus F_{n-1}$ is a nowhere dense set in $X$, so it contains no nonempty open subset. Also $f^{-1}(\{1\}) \cap U = f^{-1}((1/2,3/2)) \cap U = F_1$ is a nowhere dense set in $X$.

    We prove that $f$ is fragmentable. 
    Let $\varepsilon > 0$ and  $F$ be a nonempty closed set in $X$. If $F \subseteq X \setminus U$, then 
    ${\rm diam}f(X \cap F) = 0 < \varepsilon$.
    Suppose $F \cap U \ne \emptyset$. If $F \cap U \cap F_n \ne \emptyset$ for infinitely many $n \in \Bbb N$, 
    let $k \in \Bbb N$ be such that $2/k < \varepsilon$. 
    Put  $G =  U \setminus F_{k-1}$. 
    $G$ is a nonempty  open set in $X$ and 
    ${\rm diam}f(G \cap F) < 2/k < \epsilon$.
    Suppose there is $n \in \Bbb N$ such that  $F \cap U \subseteq F_n$. 
    Without loss of generality we can suppose that $n \in \Bbb N$ is minimal such that $F \cap U \subseteq F_n$. 
    If $n = 1$, then ${\rm diam}f(F \cap U) = 0  < \varepsilon$. 
    If $n > 1$,  the set $U \setminus F_{n-1}$ is a nonempty  open set in $X$ and ${\rm diam}f(U \setminus F_{n-1}) = 0 < \varepsilon$.

    To see that $f$ has the Lebesgue property it is enough to observe that for a perfectly normal $X$ the set $F_n\setminus F_{n-1}$ is $F_\sigma$ for each $n$.
 \end{proof}

  \begin{proposition}\label{prop:baire-frag} 
    Let $X$ be a topological space. The following are equivalent:
 	\begin{enumerate}
 		\item  $X$ is a Baire space;\label{prop:baire-cliq:1}
 		\item every cliquish $f\colon X \to \mathbb R$ is pointwise discontinuous.\label{prop:baire-cliq:2}
 	\end{enumerate}
    Moreover, if $X$ is perfectly normal, the above conditions are equivalent to:
 	\begin{enumerate}[start=3]
        \item every $f\colon X \to \mathbb R$ with the Lebesgue property is pointwise discontinuous.\label{prop:baire-cliq:3}
 	\end{enumerate}
  \end{proposition}

  \begin{proof}
      Implications ``(\ref{prop:baire-cliq:1})$\Rightarrow$(\ref{prop:baire-cliq:2})''
      and ``(\ref{prop:baire-cliq:1})$\Rightarrow$(\ref{prop:baire-cliq:3})''
      are covered by Theorem~\ref{thm:topX}. To see the reverse implications 
      suppose that $X$ is not Baire and use Lemma~\ref{le:lubica}
      to find a function which is cliquish (has the Lebesgue property, respectively)
      and is not pointwise discontinuous.
  \end{proof}

Proposition~\ref{prop:baire-frag} has its ``hereditary'' variant:
  
 \begin{proposition}\label{prop:hbaire-frag} 
 	Let $X$ be a topological space. The following are equivalent:
 	\begin{enumerate}
 		\item $X$ is a hereditarily  Baire space;\label{prop:hbaire-frag:1}
 		\item if $f\colon X \to \Bbb R$ is fragmentable,  then $f$ has (PCP).\label{prop:hbaire-frag:2}
 	\end{enumerate}
    Moreover, if $X$ is perfectly normal, the above conditions are equivalent to:
 	\begin{enumerate}[start=3]
 		\item if $f\colon X \to \Bbb R$ has the Lebesgue property,  then $f$ has (PCP).\label{prop:hbaire-frag:3}
 	\end{enumerate}
 \end{proposition}
 \begin{proof} 
      Implications ``(\ref{prop:hbaire-frag:1})$\Rightarrow$(\ref{prop:hbaire-frag:2})''
      and ``(\ref{prop:hbaire-frag:1})$\Rightarrow$(\ref{prop:hbaire-frag:3})''
      are covered by Theorem~\ref{thm:topX}. 
 
 ``(\ref{prop:hbaire-frag:2})$\Rightarrow$(\ref{prop:hbaire-frag:1})''. Suppose that $X$ is not a hereditarily Baire space. There is a closed nonempty set $F$ in $X$ such that $F$ is not a Baire space. By Lemma \ref{le:lubica} there is a nonempty relatively open set $U$ in   $F$ and a fragmentable function $f\colon F \to \Bbb R$ such that $f$ has no point of continuity in $U$ and $f(F) \subseteq \{1/n: n \in \Bbb N\}$. Define a function $g\colon X \to \Bbb R$ as follows: $g(x) = f(x)$, if $x \in F$ and $g(x) = 1$ otherwise. Then 
 $g\restriction_{\overline U}$ has no point of continuity.

    We show that $g$ is fragmentable.  
    Let $\varepsilon > 0$ and $C$ be a closed nonempty subset of $X$. 
    If $C \cap (X \setminus F) \ne \emptyset$, then ${\rm diam}g(C \cap (X \setminus F)) = 0 < \varepsilon$.

    Suppose now that $C \subseteq F$. Since the  function $f\colon F \to \Bbb R$ is fragmentable 
    there is a nonempty open set $U$ in $F$ such that $C \cap U \ne \emptyset$ and ${\rm diam}f(C \cap U) < \varepsilon$. 
    There is an open set $G$ in $X$ such that $U = G \cap F$. 
    Thus we have $G \cap C = G \cap C \cap F$ and ${\rm diam}g(G \cap C) ={\rm diam}f(C \cap U) < \varepsilon$.

    ``(\ref{prop:hbaire-frag:3})$\Rightarrow$(\ref{prop:hbaire-frag:1})''.
    It is enough to observe that the function $g$ defined in the previous paragraphs 
    has the Lebesgue property (since $X$ is perfectly normal, the set $X\setminus F$ is $F_\sigma$).
 \end{proof}
 
 Theorem \ref{thm:metricX} and Proposition \ref{prop:hbaire-frag} imply the following result.
 \begin{corollary}\label{Cor5}
  Let $X$ be a metric space. The following are equivalent:
 	\begin{enumerate}
 		\item $X$ is a hereditarily  Baire space;
 		\item  if $f\colon X \to \Bbb R$ is Borel 1,  then $f$ has (PCP).
 	\end{enumerate}
 \end{corollary}

 \begin{remark}
 Observe that, in the additional assertions of Lemma \ref{le:lubica} and Proposition \ref{prop:hbaire-frag}, we can only assume that every open set in $X$ is $F_\sigma$ instead of supposing that $X$ is perfectly normal.
 \end{remark}


 \section{Equi-families of functions}
 \label{sec:families}
 
 For collections of functions investigated in Section~\ref{sec:properties}, we study the respective subfamilies that fulfill the defining condition in common, the so-called equi-families with the adequate property. If we consider the respective orbit map with values in the product space, its property is equivalent to the respective equi-property for a given subfamily.
 
 We then present solutions of two problems from \cite{BKS}. Finally, we generalize the result of Grande \cite{Gr} showing a sufficient condition for $f\colon X\times Y\to Z$ to be a Borel function of class $\alpha$, in the language of $x$-sections and $y$-sections of $f$.
 
	\subsection{Equi-families and the orbit function}
	The notion of an equi-Baire~1 family introduced by \cite{Le} was a starting point in the previous papers \cite{BKS} and \cite{BHH}. Let $(X,\rho)$ and $(Y,d)$ be metric spaces and $\mathcal F\subseteq Y^X$. Then $\mathcal F$ is called \emph{equi-Baire~1} if 
    \begin{itemize}
        \item 
  for every $\varepsilon>0$ there is a positive function $\delta_\varepsilon\colon X\to\mathbb R_+$ such that for all $x,y\in X$ and all $f\in\mathcal F$
		\begin{equation*} 
			\varrho(x,y)<\min\left\{\delta_\varepsilon(x),\delta_\varepsilon(y)\right\} \,\,\,\Longrightarrow\,\,\,  d(f(x),f(y))<\varepsilon.
		\end{equation*}
	\end{itemize}
    In other words, a modification of (LTZ) property holds where $\delta_\varepsilon$ is chosen the same for every $f\in\mathcal F$.

    In the same spirit, we can modify several properties considered in Section~2 for a single function $f$ from a topological space $X$ to a metric space $(Y,d)$, and we obtain the respective equi-properties for a family 
    $\mathcal F\subseteq Y^X$. We mention that equi-Lebesgue 
    and equi-cliquish families were considered in \cite{BHH}.
    
	\begin{definition} \label{defi}
		Let $\mathcal F\subseteq Y^X$ be a family of functions between a topological space $X$ and a metric space $(Y,d)$. We say that     
		\begin{itemize}
			\item $\mathcal F$ is \emph{equi-weakly separated}, if for every $\varepsilon>0$ there exists a neighborhood assignment $\{V_x\}_{x\in X}$ such that for all $x,y\in X$ and all $f\in\mathcal F$
				\begin{equation*} 
					(x,y)\in V_{y}\times V_{x} \,\,\,\Longrightarrow\,\,\, d(f(x),f(y))<\varepsilon;
				\end{equation*}

			\item $\mathcal F$ is \emph{equi-fragmentable}, if
				 for every  $\varepsilon>0$ and every closed set $\emptyset\ne F\subseteq X$ there exists an open set $U$ such that $U\cap F\neq\emptyset$ and ${\rm diam} f(U\cap F)<\varepsilon$ for all $f\in\mathcal F$;
			
		    \item $\mathcal F$ is \emph{equi-cliquish}, if  
		     for every  $\varepsilon>0$ and every  open set $\emptyset\ne U\subseteq X$ there exists an open set $\emptyset\ne O\subseteq U$ such that  ${\rm diam} f(O)<\varepsilon$ for all $f\in\mathcal F$;

			\item $\mathcal F$ is \emph{equi-Lebesgue}, if 
			for every $\varepsilon >0$ there is a cover 
				$(X_n)_{n\in\mathbb N}$ of $X$ consisting of closed sets such that ${\rm diam}f(X_n)\le\varepsilon$ for all $n\in\mathbb N$
                and all $f\in\mathcal F$;

			\item $\mathcal F$ is \emph{equi-GLP}, if
			for every $\varepsilon >0$ there is a $\sigma$-discrete family $\mathcal A_\varepsilon$ consisting of closed subsets of $X$ such that $X=\bigcup \mathcal A_\varepsilon$ and ${\rm diam}f(A)\le\varepsilon$ for all $A\in\mathcal A_\varepsilon$ and all $f\in\mathcal F$;

			\item $\mathcal F$ has the \emph{point of equicontinuity property} (PECP), if for every closed set $\emptyset\ne F\subseteq X$ the family $\{f\restriction_F\colon f\in\mathcal F\}$ has a point of equicontinuity.		
		\end{itemize}
	\end{definition}

    \begin{remark}\label{nowa}
Clearly, if $X$ is a metric space in the above definition, the notions
of an equi-weakly separated family and an equi-Baire~1 family coincide.
    \end{remark}
	
	Let $X$ be a topological space, and $(Y,d)$ be a bounded metric space. For a family $\mathscr F\subseteq Y^X$ of functions, we denote by 
	$$
	f^{\sharp}_{\mathscr F}(x)=(f(x))_{f\in\mathscr F}
	$$
	the {\it orbit function} $f^\sharp_{\mathscr F}\colon X\to Y^{T}$, where $T=|\mathscr F|$ (see e.g.~\cite[Sec.~4]{MR1849204}).
	Assume that $Z=Y^T$ is equipped with the supremum metric 
	$$
	\varrho(z_1,z_2)=\sup\limits_{t\in T}d(z_1(t),z_2(t)).
	$$
It is easy to see that the following observation is valid.
This is an extended version of Proposition~1 from \cite{Kar}.
	\begin{proposition}\label{prop:orbit}
		Let $X$ be a topological space, and $(Y,d)$ be a bounded metric space. Then
		\begin{enumerate}
			\item[(1)] $\mathscr F$ is equi-continuous at $x\in X$ (on $X$) if and only if 
			$f^{\sharp}_{\mathscr F}\colon X\to (Z,\varrho)$ is continuous at  $x$ (on $X$);
				
			\item[(2)] $\mathscr F$ is equi-weakly separated if and only if 
			$f^{\sharp}_{\mathscr F}\colon X\to (Z,\varrho)$ is weakly separated;

            \item[(3)] if $X$ is metric, then  	  $\mathscr F$ is equi-Baire~1 if and only if  
			$f^{\sharp}_{\mathscr F}\colon X\to (Z,\varrho)$ has LTZ-property;

			\item[(4)] $\mathscr F$ is equi-fragmentable if and   only if  
			$f^{\sharp}_{\mathscr F}\colon X\to (Z,\varrho)$ is fragmentable;
			
			\item[(5)] $\mathscr F$ is equi-cliquish if and   only if  
			$f^{\sharp}_{\mathscr F}\colon X\to (Z,\varrho)$ is cliquish;
			
			\item[(6)]	$\mathscr F$ is equi-Lebesgue if and only if  
			$f^{\sharp}_{\mathscr F}\colon X\to (Z,\varrho)$ has the Lebesgue property;
			
			\item[(7)] $\mathscr F$ is equi-GLP if and only if 
			$f^{\sharp}_{\mathscr F}\colon X\to (Z,\varrho)$ has the generalized Lebesgue property;

            \item[(8)] $\mathscr F$ has (PECP) if and   only if  
			$f^{\sharp}_{\mathscr F}\colon X\to (Z,\varrho)$ has (PCP).
		\end{enumerate}	  
	\end{proposition}

    We will show some applications of this proposition.

    \begin{remark} \label{remi}
    Let $X$ be a topological space and $(Y,d)$ be a metric space. Put $d^*:=\min\{d,1\}$. Let $\mathcal F\subseteq Y^X$.
    Clearly, $\mathcal F$ has any equi-property from 
    Definition~\ref{defi} using the metric $d$ if and only if it has the respective equi-property using the metric $d^*$. Similarly,
    if $X$ is a metric space, $\mathcal F$ is equi-Baire~1  using  the metric $d$ if and only if it has equi-Baire~1 property using the metric $d^*$.
    \end{remark}

    \begin{proposition} \label{propa1}
    Let $X$ and $Y$ be metric spaces and $\mathcal F\subseteq Y^X$. \begin{enumerate}
        \item If $\mathcal F$ has (PECP), then $\mathcal F$ is equi-Baire~1.
        \item Assume that $X$ is hereditary Baire. Then 
        $\mathcal F$ is equi-Baire~1 if and only if $\mathcal F$ has (PECP).
    \end{enumerate}
    \end{proposition}
    \begin{proof}
        It is enough to apply Theorem \ref{thm:metricX}, 
        Proposition \ref{prop:orbit} and Remark \ref{remi}. Note that (2) generalizes the result of \cite[Proposition 32]{Le}.
    \end{proof}

    The following proposition extends Proposition~6.4 from \cite{BHH}.

    \begin{proposition} \label{propa2}
      Let $X$ be a Baire topological space, $Y$ be a metric space, and $\mathcal F\subseteq Y^X$. 
      \begin{enumerate} 
      \item If $\mathcal F$ is equi-GLP, then $\mathcal F$ is equi-cliquish.
      \item If $X$ is a metric space and $\mathcal F$ is equi-Baire~1, then $\mathcal F$ is equi-cliquish.
      \end{enumerate}
    \end{proposition}
    \begin{proof}
    Assertion (1) follows from Theorem \ref{thm:topX}, Proposition \ref{prop:orbit} and Remark \ref{remi}. Assertion (2) follows from Theorem \ref{thm:metricX}, Proposition \ref{prop:orbit} and Remark \ref{remi}.
    \end{proof}

    For a family $\mathcal F$ of functions from a topological space $X$ to a metric space $Y$ we denote by $EC(\mathcal F)$ the set of equi-continuity points of $\mathcal F$. It is known that this set is $G_\delta$ since $EC=\bigcap_{n\in\mathbb N}O(1/n)$ where
    $$O(\varepsilon)=\bigcup\{U\colon U\subseteq X,\; U \text{ nonempty open, diam}f(U)<\varepsilon\text{ for all }f\in\mathcal F\}$$
    is open for each $\varepsilon>0$; see \cite{FHT}.
    
We can generalize Proposition 6.3 from \cite{BHH} as follows.

\begin{proposition} \label{propa3}
  Let $X$ be a Baire topological space, $Y$ be a metric space, and $\mathcal F\subseteq Y^X$. 
  \begin{enumerate}
  \item If $\mathcal F$ is equi-cliquish, then $EC(\mathcal F)$ is a dense $G_\delta$ set. Thus, if $\mathcal F$ is equi-GLP, then $EC(\mathcal F)$ is a dense $G_\delta$ set.
  \item If $X$ is a metric space and $\mathcal F$ is equi-Baire~1, then $EC(\mathcal F)$ is a dense $G_\delta$ set.
  \end{enumerate}
\end{proposition}
\begin{proof}
    Assertion (1) follows from Theorem \ref{thm:topX}, Proposition \ref{prop:orbit} and Remark \ref{remi}. Assertion (2) follows from Theorem \ref{thm:metricX}, Proposition \ref{prop:orbit} and Remark \ref{remi}.
    \end{proof}

\subsection{Closures in the topology of pointwise convergence}

The following result is classical.

\begin{proposition} (\cite{Kel}, Theorem 14, Chapter 7)  Let $X$ be a topological space, $(Y,d)$ be a metric space and  $\mathcal F$ be a family of functions from $X$ to $Y$.  If $\mathcal F$ is equicontinuous at $x$, then the closure $\overline{\mathcal F}$ of $\mathcal F$
relative to the topology of pointwise convergence is also equicontinuous at $x$.
\end{proposition}

We can state similar results for other equi-families.

\begin{proposition} (\cite{Le}, Proposition 35) Let $(X,d_X)$  and $(Y,d_Y)$ be  metric spaces.  If $\mathcal F$ is an equi-Baire 1 family of functions from $X$ to $Y$, then the closure $\overline{\mathcal F}$ of $\mathcal F$
relative to the topology of pointwise convergence is equi-Baire 1 too.
\end{proposition}

\begin{proposition} (\cite{BHH}, Proposition 4.1) Let $(X,d_X)$  and $(Y,d_Y)$ be  metric spaces.  If $\mathcal F$ is an equi-Lebesgue family of functions from $X$ to $Y$, then the closure $\overline{\mathcal F}$ of $\mathcal F$
relative to the topology of pointwise convergence is equi-Lebesgue too.
\end{proposition}

\begin{proposition}
Let $X$ be a topological space, and let $(Y,d)$ be a metric space.  If $\mathcal F$ is an equi-weakly separated family of functions from $X$ to $Y$, then the closure $\overline{\mathcal F}$ of $\mathcal F$
relative to the topology of pointwise convergence is also equi-weakly separated.
\end{proposition}
\begin{proof} Let $\varepsilon > 0$. Since $\mathcal F$ is an equi-weakly separated family of functions from $X$ to $Y$, there exists a neighbourhood assignment $(V_x)_{x \in X}$ such that for every $x, y \in X$ and $f \in \mathcal F$
$$(x,y) \in V_y \times V_x \textup{\ implies\ } d(f(x),f(y)) \le \varepsilon/3.$$
Then for $\varepsilon$ and   the neighbourhood assignment $(V_x)_{x \in X}$ we have: for every $x, y \in X$ and for every $g \in \overline{\mathcal F}$
$$(x,y) \in V_y \times V_x \textup{\ implies\ } d(g(x),g(y)) \le \varepsilon.$$
Let $(x,y) \in V_y \times V_x$. If $g \in \mathcal F$,  then $d(g(x),g(y)) \le \varepsilon/3$. Let  $g \in \overline{\mathcal F} \setminus \mathcal F$. There is a function $f \in \mathcal F$ such that $d(f(x),g(x)) < \varepsilon/3$ and $d(f(y),g(y))) < \varepsilon/3$. 
Then we have
$$d(g(x),g(y)) \le d(g(x),f(x)) + d(f(x),f(y)) + d(f(y),g(y)) \le \varepsilon.$$
\end{proof}

Using similar ideas, we can also prove the following implications.

\begin{proposition} Let $X$ be a topological space and $(Y,d)$ be a metric space.
\begin{enumerate}
\item
If $\mathcal F$ is an equi-Lebesgue  family of functions from $X$ to $Y$, then the closure $\overline{\mathcal F}$ of $\mathcal F$ relative to the topology of pointwise convergence is equi-Lebesgue  too.
\item
If $\mathcal F$ is an equi-GLP  family of functions from $X$ to $Y$, then the closure $\overline{\mathcal F}$ of $\mathcal F$ relative to the topology of pointwise convergence is also equi-GLP.
\item
If $\mathcal F$ is an equi-cliquish  family of functions from $X$ to $Y$, then the closure $\overline{\mathcal F}$ of $\mathcal F$ relative to the topology of pointwise convergence is equi-cliquish too.
\item
If $\mathcal F$ is an equi-fragmentable  family of functions from $X$ to $Y$, then the closure $\overline{\mathcal F}$ of $\mathcal F$ relative to the topology of pointwise convergence is equi-fragmentable  too.
\item
If $\mathcal F$  has (PECP), then the closure $\overline{\mathcal F}$ of $\mathcal F$ relative to the topology of pointwise convergence  has (PECP)  too.
\end{enumerate}
\end{proposition}

 
\subsection{Applications to functions defined on $X\times Y$}

We are going to solve two problems from the earlier paper \cite{BKS}.

\begin{problem}\cite[Prob.~3.7]{BKS}
Suppose that $\mathcal F\subseteq \mathbb{R}^{X\times Y}$ 
is a family of separately equi-continuous functions defined on the product of Polish spaces $X$ and $Y$, 
i.e.
\begin{itemize}
\item $\{f(x,\cdot)\colon f\in\mathcal F\}$ is equi-continuous for each $x\in X$, and 
\item $\{f(\cdot,y)\colon f\in\mathcal F\}$ is equi-continuous for each $y\in Y$.
\end{itemize}
Is $\mathcal F$ equi-Baire 1?
\end{problem}

Let us give a positive solution in a general setting.
Bouziad \cite[Proposition 4.3]{Bou} proved the following fact. If $X,Y$
are topological spaces and $Z$ is a metric space, then every separately continuous function $f\colon X\times Y\to Z$ is weakly separated. Now, we obtain the following corollary which solves the problem.

\begin{corollary} \label{T1}
Let $X,Y$ be topological spaces and $Z$ be a metric space. Let $\mc F\subseteq Z^{X\times Y}$ be a family of separately equi-continuous functions defined on $X\times Y$. Then $\mc F$ is equi-weakly separated.
\end{corollary}
\begin{proof}
Consider the orbit function $f^\sharp_{\mathscr F}\colon X\times Y\to Z^{T}$ (where $T:=|\mathscr F|$) as in Proposition \ref{prop:orbit}. Note that it is separately continuous. By the Bouziad result, $f^\sharp_{\mathscr F}$
is weakly separated and we infer that ${\mathscr F}$ is equi-weakly separated by Proposition \ref{prop:orbit}.
\end{proof}


Let us state the next problem \cite[Prob.~3.11]{BKS} in a strictly formal version as follows.

\begin{problem}\label{prob2}
Suppose that $X$ and $Y$ are Polish and for each $n\in\mathbb{N}$ we have $f_n\colon X\times Y\to\mathbb{R}$.
Moreover, suppose that:
\begin{itemize}
\item the family $\{f_n(x,\cdot) \colon n\in\mathbb{N}\}$ is equi-Baire 1 for each $x\in X$, and
\item the family $\{f_n(\cdot,y) \colon n\in\mathbb{N}\}$ is equi-Baire 1 for each $y\in Y$.
\end{itemize}
Does there exist a comeager $G_\delta$ set $A\subseteq X$ such that $\{f_n|_{A\times Y}\colon n\in\mathbb{N}\}$ is jointly equi-Baire 1?
\end{problem}

\begin{remark} \label{rem:Gr}
Using Proposition~\ref{prop:orbit}(6) for the families:
\begin{itemize}
\item $\mathcal{F}:=\{f_n\colon n\in\mathbb{N}\}$,
\item $\mathcal{F}^1_{x}:=\{f_n(x,\cdot) \colon n\in\mathbb{N}\}$ for each $x\in X$, 
\item $\mathcal{F}^2_{y}:=\{f_n(\cdot,y) \colon n\in\mathbb{N}\}$ for each $y\in Y$,
\end{itemize}
we can see that the above problem is equivalent to the question whether for each separately Baire~1
function $f^{\#}\colon X\times Y\to \mathbb{R}^\mathbb{N}$ there exists 
a comeager $G_\delta$ set $A\subseteq X$ such that $f^{\#}|_{A\times Y}$ is jointly Baire~1.
\end{remark}

Let us show a negative solution of Problem \ref{prob2}. Grande noticed that there exists a function $f\colon\R^2\to\R$ that has all $x$-sections and all $y$-sections of class Baire 1 but $f$ is not
Lebesgue measurable (or it does not have the Baire property); see e.g. \cite[Remark 2]{Gr}.
One can use this fact to obtain the respective counterexample for Problem \ref{prob2}. However, we propose a different simple construction.
Let $B\subseteq\R$ be a Bernstein nonmeasurable set. It is known that for any uncountable Borel set $A\subseteq\R$ the intersection $B\cap A$ is not a Borel set. Let 
$g\colon\R^2\to\R$ be the characteristic function of the set $B^*=\{(x,x)\colon x\in B\}$.
Then $g$ is separately Baire~1. Suppose $A\subseteq\R$ is a nonmeager $G_\delta$ set such that
$g\restriction_{A\times\R}$ is Baire~1. Then $A\cap B$ should be Borel, a contradiction.
Similarly, the function $f\colon\R^2\to\R^\N$ given by $f=(g,g,\dots)$ is separately Baire~1
and $f\restriction_{A\times\R}$ is not jointly Baire~1 for every nonmeager $G_\delta$ set $A\subseteq\R$.
This by Remark \ref{rem:Gr} leads to the negative solution of Problem \ref{prob2}.

Grande proved in \cite[Thm 6]{Gr} the following interesting fact on real functions of two variables.
If $f\colon\R^2\to\R$ is a function with all $x$-sections that form a family having (PECP) and all $y$-sections of Borel class 
$\alpha$, then $f$ is of Borel class $\alpha$. We have observed that this result with the same proof remains valid if $f\colon X\times Y\to Z$ where $X,Y,Z$ are metric spaces and $Y$ is separable.
Below, we propose a further extension of this theorem.

\begin{definition}\label{alphaGLP1}
		We say that a function $f\colon X\to Y$ between a topological space $X$ and a metric space $Y$ has  \emph{$\alpha$-generalized Lebesgue property}, for $1\le \alpha<\omega_1$, if 
			\begin{itemize}
				\item[{\rm ($\alpha$-GLP)}] for every $\varepsilon >0$ there is a $\sigma$-discrete family $\mathcal A_\varepsilon$ in $X$ consisting of Borel sets of additive class $\alpha$  such that $X=\bigcup \mathcal A_\varepsilon$ and ${\rm diam}f(A)\le\varepsilon$ for all $A\in\mathcal A_\varepsilon$.
			\end{itemize}
	\end{definition}
	
	\begin{remark}
		It is easy to see that for an isolated ordinal $\alpha=\beta+1$ we may assume that all elements from $\mathcal A_\varepsilon$ belongs to the multiplicative class $\beta$. Hence, Definition \ref{alphaGLP1} for $\alpha=1$ is consistent with the definition of the generalized Lebesgue property. 
		
		If $\alpha$ is limit, then we may assume that all elements of $\mathcal A$ belongs to multiplicative classes less than $\alpha$. Since every set of multiplicative class less than $\alpha$ in a metric space is ambiguous of class $\alpha$, one can substitute the word ``additive'' by ``ambiguous'' in Definition \ref{alphaGLP1}.
	\end{remark}

\begin{definition}
	Let $\mathscr F\subseteq Y^X$ be a family of functions between $X$ and $Y$, and $1\le\alpha<\omega_1$. We say that $\mathscr F$ {\it is equi-$\alpha$-GLP}, if ($\alpha$-GLP) holds for every $f\in\mathcal F$, i.e.
    for every $\varepsilon >0$ there is a $\sigma$-discrete family $\mathcal A_\varepsilon$ in $X$ consisting of Borel sets of additive class $\alpha$  such that $X=\bigcup \mathcal A_\varepsilon$ and ${\rm diam}f(A)\le\varepsilon$ for all $A\in\mathcal A_\varepsilon$ and $f\in\mathcal{F}$.
\end{definition}

	If $f\colon X\times Y\to Z$, then for $x\in X$ and $y\in Y$ by $f^x$ and $f_y$ we denote $x$-section  $f^x\colon Y\to Z$ and $y$-section $f_y\colon X\to Z$  of the function $f$, respectively.
	
	 Let $\mathcal A$ and $\mathcal B$ be families of sets. We write $\mathcal A\prec\mathcal B$ if for every set $A\in\mathcal A$ there exists a set $B\in\mathcal B$ such that $A\subseteq B$.
	
 \begin{theorem}
 Let $X, Y$ and $(Z,d)$ be metric spaces, $1\le\alpha<\omega_1$ and let a function $f\colon X\times Y\to Z$ satisfy the following properties:
 \begin{enumerate}
 	\item the family $\mathscr F_X=\{f^x: x\in X\}$ is equi-$\alpha$-GLP,
 	
 	\item for every $y\in Y$ the section $f_y\colon X\to Z$ belongs to the Borel class $\alpha$. 
 \end{enumerate}
 
 Then $f\colon X\times Y\to Z$ belongs to the  Borel class $\alpha$.
\end{theorem}

\begin{proof} We will show that for every $\varepsilon>0$ there exists a function $g\colon X\times Y\to Z$ of Borel class $\alpha$ such that
	\begin{center}
		 $d(f(p),g(p))\le\varepsilon$ for all $p\in X\times Y$.
	\end{center} 
	Fix $\varepsilon>0$. Since $\mathscr F_X$ is equi-$\alpha$-GLP, there exists a $\sigma$-discrete family $\mathcal A$ in $Y$ consisting of sets of additive and multiplicative class $\alpha$ simultaneously such that
	\begin{center}
		  $Y=\bigcup \mathcal A$ \quad and \quad ${\rm diam}\,f^x(A)\le\varepsilon$
	\end{center} 
	for all $A\in\mathcal A$ and $x\in X$. Let $\mathcal A=\bigcup\limits_n\mathcal A_n$, where each family $\mathcal A_n$ is discrete in $Y$. For every $n\in\mathbb N$ we put 
	\begin{center}
		 	$\mathcal B_n=\{ A\setminus \bigcup\limits_{k<n}\bigcup \mathcal A_k: A\in\mathcal A_n\}$\,\,\,and\,\,\,$B_n=\bigcup\{B:B\in\mathcal B_n\}$.
	\end{center}
    Notice that each family $\mathcal B_n$ is discrete (in particular, disjoint), since $\mathcal B_n\prec\mathcal A_n$.  Let 
      $$
      \mathcal B=\bigcup\limits_{n=1}^\infty \mathcal B_n.
      $$
 	Then $\mathcal B$ is a  $\sigma$-discrete family of sets of ambiguous class $\alpha$ with $\mathcal B\prec \mathcal A$ and $\bigcup\mathcal B=Y$. Moreover, $B'\cap B''=\emptyset$ for all $B'\in\mathcal B_k$, $B''\in\mathcal B_n$ and distinct $k,n\in\mathbb N$. Notice that every set $B_n$ belongs to the additive and multiplicative class $\alpha$ simultaneously as the union of a discrete family of sets of the same classes in metric space $Y$. 
	
	Take an arbitrary $y_{B,n}\in B$ for every nonempty $B\in\mathcal B_n$, $n\in \mathbb N$. For each $y\in Y$ we put $n(y)=\min\{n\in\mathbb N: y\in B_n\}$. Now let $(x,y)\in X\times Y$. Since the family $\mathcal B_{n(y)}$ is disjoint, there exists a unique set $B(y)\in\mathcal B_{n(y)}$ such that $y\in B(y)$. We define
	$$
	g(x,y)=f(x,y_{B(y),n(y)}).
	$$
	Then for every point $(x,y)\in X\times Y$ we have 
	$$
	d(f(x,y),g(x,y))=d(f(x,y), f(x,y_{B(y),n(y)}))\le {\rm diam}\, f^x(B(y)),
	$$
	since $y, y_{B(y),n(y)}\in B(y)$. It follows from the property $\mathcal B\prec\mathcal A$ that ${\rm diam}\, f^x(B(y))\le {\rm diam}\, f^x(A)$ for some $A\in\mathcal A$. Hence,
	$$
		d(f(x,y),g(x,y))\le\varepsilon.
	$$

It remains to prove that $g\colon X\times Y\to Z$ belongs to the Borel class $\alpha$. Fix $n\in\mathbb N$ and let 
$g_n=g\restriction_{X\times B_n}$. It follows from the definition of $g$ and assumption~(2) of the theorem, that $g_{n,B}=g_n\restriction_{X\times B}$ belongs to the Borel class $\alpha$ for every $B\in\mathcal B_n$. Then for every open set $W\subseteq Z$ we have that the preimage
$$
g_n^{-1}(W)=\bigcup\{g_{n,B}^{-1}(W):B\in\mathcal B_n\}
$$
is of the additive class $\alpha$ in $X\times B_n$ as a union of a discrete family of sets of the same class. Now, since $\{X\times B_n:n\in\mathbb N\}$ is a countable covering of $X\times Y$ by sets of ambiguous class $\alpha$ and $g\restriction_{X\times B_n}$ is Borel $\alpha$, we conclude that $g\colon X\times Y\to Z$ belongs to the Borel class $\alpha$.

	Finally, let us remember that the family of functions of Borel class $\alpha$ between metric spaces is closed under uniform limits.	 
\end{proof}

By Proposition~\ref{prop:orbit} and Theorem~\ref{thm:metricX}:
$$\textrm{(PECP)} \Rightarrow \textrm{(equi-fragmentable)} \Rightarrow \textrm{(equi-Baire 1)} \Rightarrow \textrm{(equi-GLP)} \Rightarrow \textrm{(equi-$\alpha$-GLP)},$$
so the following corollary is a generalization of Grande's result. 

  \begin{corollary}
  	Let $X$, $Y$ and $Z$ be metric spaces and $f\colon X\times Y\to Z$ satisfies the following properties
  	\begin{enumerate}
  		\item the family $\mathscr F_X=\{f^x: x\in X\}$ has (PECP),
  		
  		\item for every $y\in Y$ the section $f_y\colon X\to Z$ belongs to the Borel class $\alpha$. 
  	\end{enumerate}
  	
  	Then $f\colon X\times Y\to Z$ belongs to the Borel class $\alpha$.
  \end{corollary}

  \noindent
  {\bf Ackonwledgement.} We would like to thank Tomek Natkaniec for showing us Theorem 6 in Grande's article \cite{Gr} which is relevant to our consideration of Problem \ref{prob2}.



\bibliographystyle{amsalpha}
\bibliography{Equi_properties}

\end{document}